\documentclass[final]{siamltex}

\newcommand{\tr}{^{\sf T}}
\newcommand{\m}[1]{{\bf{#1}}}

\newcommand{\C}[1]{{\cal {#1}}}

\usepackage{bm}
\usepackage{amsmath}
\usepackage{mathtext}
\usepackage{amsfonts}
\usepackage{amssymb}
\usepackage[mathscr]{eucal}
\usepackage{makeidx}
\usepackage{latexsym}
\usepackage{lscape}
\usepackage{fancyhdr}
\usepackage{graphicx}
\newtheorem{remark}{Remark}

\medskip

\title{{\bf Gradient-based methods for sparse recovery}
\thanks{ October 25, 2009.
This material is based upon work supported by the
National Science Foundation under Grant 0619080.
}}
\date{}

\author{
        William W. Hager\thanks{{\tt hager@math.ufl.edu},
        http://www.math.ufl.edu/$\sim$hager,
        PO Box 118105,
        Department of Mathematics,
        University of Florida, Gainesville, FL 32611-8105.
        Phone (352) 392-0281. Fax (352) 392-8357.}
\and
        Dzung T. Phan\thanks{{\tt dphan@math.ufl.edu},
        http://www.math.ufl.edu/$\sim$dphan,
        PO Box 118105,
        Department of Mathematics,
        University of Florida, Gainesville, FL 32611-8105.
        Phone (352) 392-0281. Fax (352) 392-8357.}
\and
        Hongchao Zhang\thanks{{\tt hozhang@math.lsu.edu},
        http://www.math.lsu.edu/$\sim$hozhang,
        Department of Mathematics,
        140 Lockett Hall,
        Center for Computation and Technology,
        Louisiana State University, Baton Rouge, LA 70803-4918.
        Phone (225) 578-1982. Fax (225) 578-4276.}
}

\begin{document}
\maketitle

\begin{abstract}
The convergence rate is analyzed for the SpaSRA algorithm
(Sparse Reconstruction by Separable Approximation)
for minimizing a sum $f (\m{x}) + \psi (\m{x})$ where $f$ is
smooth and $\psi$ is convex, but possibly nonsmooth.
It is shown that if $f$ is convex, then the error in the
objective function at iteration $k$, for $k$ sufficiently large,
is bounded by $a/(b+k)$ for suitable choices of $a$ and $b$.
Moreover, if the objective function is strongly convex, then the
convergence is $R$-linear.
An improved version of the algorithm based on a
cycle version of the BB iteration and an adaptive line search is given.
The performance of the algorithm is investigated using applications
in the areas of signal processing and image reconstruction.
\end{abstract}



\pagestyle{myheadings} \thispagestyle{plain}
\markboth{W. W. HAGER, D. T. PHAN, H. ZHANG}
{GRADIENT-BASED METHODS FOR SPARSE RECOVERY}

\section{Introduction}
\label{introduction}

In this paper we consider the following optimization problem
\begin{equation}\label{main}
\min_{\m{x} \in \mathbb{R}^n} \;\; \phi(\m{x}) : = f(\m{x}) + \psi(\m{x}),
\end{equation}
where $f: \mathbb{R}^n \rightarrow \mathbb{R}$ is a smooth function,
and $\psi: \mathbb{R}^n \rightarrow \mathbb{R}$ is convex.
The function $\psi$, usually called the \emph{regularizer} or
\emph{regularization function}, is finite for
all $\m{x} \in \mathbb{R}^n$, but possibly nonsmooth.
An important application of (\ref{main}), found in the signal
processing literature, is the well-known $\ell_2 -
\ell_1$ problem (called \emph{basis pursuit denoising} in
\cite{chen98})
\begin{equation}\label{BPDN}
\min_{\m{x} \in \mathbb{R}^n} \;\; \frac{1}{2}\|\m{A}\m{x}-\m{b}\|^2_2 +
\tau \|\m{x}\|_1,
\end{equation}
where $\m{A} \in \mathbb{R}^{k \times n}$ (usually $k \le n$),
$\m{b} \in \mathbb{R}^k, \tau \in \mathbb{R}$, $\tau \ge 0$,
and $\| \cdot \|_1$ is the $1$-norm.

Recently, Wright, Nowak, and Figueiredo \cite{wright09a} introduced the
Sparse Reconstruction by Separable Approximation algorithm
(SpaRSA) for solving (\ref{main}).
The algorithm has been shown to work well in practice.
In \cite{wright09a} the authors establish global convergence of SpaRSA.
In this paper, we prove an estimate of the form
$a/(b+k)$ for the error in the objective function when $f$ is convex.
If the objective function is strongly convex,
then the convergence of the objective
function and the iterates is at least R-linear.
A strategy is presented for improving the performance of SpaRSA
based on a cyclic Barzilai-Borwein step \cite{dyh01, dhsz05, fmmr99, rs02}
and an adaptive choice \cite{hz05a}
for the reference function value in the line search.
The paper concludes with a series of numerical experiments
in the areas of signal processing and image reconstruction.

Throughout the paper
$\nabla f (\m{x})$ denotes the gradient of $f$, a row vector.
The gradient of $f (\m{x})$, arranged as a column vector,
is $\m{g} (\m{x})$.
The subscript $k$ often represents the iteration number
in an algorithm, and $\m{g}_k$ stands for $\m{g}(\m{x}_k)$.
$\| \cdot\|$ denotes $\| \cdot \|_2$, the Euclidean norm.
$\partial \psi (\m{y})$ is the subdifferential at $\m{y}$,
a set of row vectors.
If $\m{p} \in \partial \psi (\m{y})$, then
\[
\psi (\m{x}) \ge \psi (\m{y}) + \m{p}(\m{x} - \m{y})
\]
for all $\m{x} \in \mathbb{R}^n$.

\section{The SpaRSA algorithm}
\label{SpaSRA}

The SpaRSA algorithm, as presented in \cite{wright09a}, is as follows:
\bigskip

{\tt
\begin{tabular}{l}
{\sc Sparse Reconstruction by Separable Approximation (SpaRSA)}\\
\hline
Given $\eta > 1$, $\sigma \in (0,1)$,
$[\alpha_{\min}, \alpha_{\max}] \subset (0, \infty)$,
and starting guess $\m{x}_1$. \\
Set $k = 1$.\\
\begin{tabular}{rl}
Step 1. & Choose $\alpha_0 \in [\alpha_{\min} , \alpha_{\max} ]$\\
Step 2. & Set $\alpha = \eta^j \alpha_0$
where $j \ge 0$ is the smallest integer such that \\[.1in]
& $\phi (\m{x}_{k+1}) \le \phi_k^R - \sigma \alpha
\|\m{x}_{k+1} - \m{x}_k\|^2 $ where\\[.1in]
& $\m{x}_{k+1} = \mbox{arg} \min \{
\nabla f(\m{x}_k)\m{z} + \alpha \|\m{z} - \m{x}_k\|^2
+ \psi (\m{z}) : \m{z} \in \mathbb{R}^n\}$. \\[.1in]
Step 3. & If $\m{x}_{k+1} = \m{x}_k$, terminate. \\
Step 4. & Set $k=k+1$ and go to step 1. \\
\end{tabular}\\
\hline \\
\end{tabular}}

The parameter $\alpha_0$ in \cite{wright09a} was taken to be the
BB parameter \cite{bb88} with safeguards:
\begin{equation}\label{BB}
\alpha_0 = \alpha^{BB}_k =
\min \;\; \{ \|\alpha \m{s}_k - \m{y}_k\|:
\alpha_{\min} \le \alpha \le \alpha_{\max} \}
\end{equation}
where $\m{s}_k = \m{x}_k - \m{x}_{k-1}$ and
$\m{y}_k = \m{g}_k - \m{g}_{k-1}$.
Also, in \cite{wright09a}, the reference value $\phi_k^R$ is
the GLL \cite{gll86} reference value $\phi_k^{\max}$ defined by
\begin{equation}\label{GLL}
\phi_k^{\max} =
\max \{ \phi (\m{x}_{k-j}) : 0 \le j < \min(k,M)\} .
\end{equation}
In other words, at iteration $k$, $\phi_k^{\max}$ is the maximum
of the $M$ most recent values for the objective function.
Note that if $\m{x}_{k+1} = \m{x}_k$, then
\[
\m{0} \in \nabla f(\m{x}_k) + \partial \psi(\m{x}_{k+1}) =
\nabla f(\m{x}_{k+1}) + \partial \psi(\m{x}_{k+1}).
\]
Hence, $\m{x}_{k+1} = \m{x}_k$ is a stationary point.

The overall structure of the SpaRSA algorithm is closely related to that
of the Iterative Shrinkage Thresholding Algorithm (ISTA)
\cite{Chambolle98,Daubechies04,Figueiredo03,HaleZhang07,Vonesch07}.
ISTA, however, employs a fixed choice for $\alpha$ related to the
Lipschitz constant for $f$, while  SpaRSA employs a nonmonotone line search.
A sublinear convergence result for a monotone line search version
of ISTA is given by Beck and Teboulle \cite{beck09} and by
Nesterov \cite{nesterov07}.
In Section \ref{convergence}
we give a sublinear convergence result for the nonmonotone SpaRSA,
while Section \ref{strong_convexity} gives a linear convergence
result when the objective function is strongly convex.

In \cite{wright09a} it is shown that the line search in Step 2
terminates for a finite $j$ when $f$ is Lipschitz continuously differentiable.
Here we weaken this condition by only requiring Lipschitz continuity
over a bounded set.
\begin{proposition}
\label{StepProposition}
Let $\C{L}$ be the level set defined by
\begin{equation}\label{L}
\C{L} = \{ \m{x} \in \mathbb{R}^n : \phi (\m{x}) \le \phi(\m{x}_1) \} .
\end{equation}
We make the following assumptions:
\begin{itemize}
\item[{\rm (A1)}]
The level set $\C{L}$ is contained in
the interior of a compact, convex set $\C{K}$,
and $f$ is Lipschitz continuously differentiable on $\C{K}$.
\item[{\rm (A2)}]
$\psi$ is convex and $\psi(\m{x})$ is finite for all $\m{x} \in \mathbb{R}^n$.
\end{itemize}
If $\phi (\m{x}_k) \le \phi_k^R \le \phi(\m{x}_1)$,
then there exists $\bar{\alpha}$ with the property that
\[
\phi (\m{x}_{k+1}) \le \phi_k^R - \sigma \alpha
\|\m{x}_{k+1} -\m{x}_k\|^2
\]
whenever $\alpha \ge \bar{\alpha}$ where $\m{x}_{k+1}$ is
obtained as in Step $2$ of SpaRSA.
\end{proposition}

\begin{proof}
Let $\Phi_k$ be defined by
\[
\Phi_k (\m{z}) = f(\m{x}_k) + \nabla f(\m{x}_k)(\m{z}-\m{x}_k)
+ \alpha \|\m{z} - \m{x}_k\|^2 + \psi (\m{z}) ,
\]
where $\alpha \ge 0$.
Since $\Phi_k$ is a strongly convex quadratic, its level sets are
compact, and the minimizer $\m{x}_{k+1}$ in Step 2 exists.
Since $\m{x}_{k+1}$ is the minimizer of $\Phi_k$, we have
\begin{eqnarray*}
\Phi_k(\m{x}_{k+1}) &=& f(\m{x}_k) + \nabla f(\m{x}_k)(\m{x}_{k+1}-\m{x}_k) +
\alpha \|\m{x}_{k+1} - \m{x}_k\|^2
+ \psi (\m{x}_{k+1}) \\
&\le& \Phi_k(\m{x}_k) = f(\m{x}_k) + \psi(\m{x}_k) .
\end{eqnarray*}
This is rearranged to obtain
\begin{eqnarray*}
\alpha \|\m{x}_{k+1} - \m{x}_k\|^2 &\le&
\nabla f(\m{x}_k)(\m{x}_{k}-\m{x}_{k+1}) +
\psi(\m{x}_k) - \psi (\m{x}_{k+1}) \\
&\le &
\nabla f(\m{x}_k)(\m{x}_{k}-\m{x}_{k+1}) +
\m{p}_k(\m{x}_k - \m{x}_{k+1}) ,
\end{eqnarray*}
where $\m{p}_k \in \partial \psi (\m{x}_k)$.
Taking norms yields
\begin{equation}\label{x_k+1_bound}
\|\m{x}_{k+1} - \m{x}_k\| \le (\|\m{g}_k\| + \|\m{p}_k\|)/\alpha .
\end{equation}
By Theorem 23.4 and Corollary 24.5.1 in \cite{Rockafellar70} and by
the compactness of $\C{L}$, there exists a constant $c$,
independent of $\m{x}_k \in \C{L}$, such that $\|\m{g}_k\| + \|\m{p}_k\| \le c$.
Consequently, we have
\[
\|\m{x}_{k+1} - \m{x}_k\| \le c/\alpha .
\]
Since $\C{K}$ is compact and $\C{L}$ lies in the interior
of $\C{K}$, the distance $\delta$ from $\C{L}$
to the boundary of $\C{K}$ is positive.
Choose $\beta \in (0, \infty)$ so that $c/\beta \le \delta$.
Hence, when $\alpha \ge \beta$, $\m{x}_{k+1} \in \C{K}$
since $\m{x}_k \in \C{L}$.

Let $\lambda$ denote the Lipschitz constant for $f$ on $\C{K}$ and
suppose that $\alpha \ge \beta$.
Since $\m{x}_k \in \C{L} \subset \C{K}$
and $\|\m{x}_{k+1} - \m{x}_k\| \le \delta$, we have $\m{x}_{k+1} \in \C{K}$.
Moreover, due to the convexity of $\C{K}$,
the line segment connecting $\m{x}_k$ and $\m{x}_{k+1}$ lies in $\C{K}$.
Proceeding as in \cite{wright09a}, a Taylor expansion around
$\m{x}_k$ yields
\[
f(\m{x}_{k+1}) \le f(\m{x}_k)  +
\nabla f (\m{x}_k) (\m{x}_{k+1} - \m{x}_k) +
.5\lambda\|\m{x}_{k+1} - \m{x}_k\|^2 .
\]
Adding $\psi(\m{x}_{k+1})$ to both sides, we have
\begin{eqnarray}
\phi(\m{x}_{k+1}) &\le& \Phi_k (\m{x}_{k+1}) +
(.5\lambda - \alpha)\|\m{x}_{k+1} - \m{x}_k\|^2 \label{phiPhi}\\
&\le& \Phi_k (\m{x}_{k}) + \left( .5\lambda - \alpha \right)
\|\m{x}_{k+1} - \m{x}_k\|^2 \nonumber \\
&=& \phi (\m{x}_{k}) + \left( .5 \lambda - \alpha \right)
\|\m{x}_{k+1} - \m{x}_k\|^2 \nonumber \\
&\le& \phi_k^R +
\left( .5\lambda - \alpha \right) \|\m{x}_{k+1} - \m{x}_k\|^2
\quad \mbox{since } \phi(\m{x}_k) \le \phi_k^R \nonumber \\
&\le& \phi_k^R
- \sigma \alpha \|\m{x}_{k+1} - \m{x}_k\|^2
\quad \mbox{if }
.5\lambda - \alpha \le -\sigma \alpha . \nonumber
\end{eqnarray}
Hence, the proposition holds with
\[
\bar{\alpha} = \max \; \left\{ \beta, \frac{\lambda}{2(1-\sigma)} \right\} .
\]
\end{proof}

\begin{remark}
\label{exist_phiR}
Suppose $\phi_k^R \le \phi(\m{x}_1)$.
In Step 2 of SpaRSA, $\m{x}_{k+1}$ is chosen so that
$\phi(\m{x}_{k+1}) \le \phi_k^R$.
Hence, there exists $\phi_{k+1}^R$ such that
$\phi(\m{x}_{k+1}) \le \phi_{k+1}^R \le \phi(\m{x}_1)$.
In other words, if the hypothesis
``$\phi (\m{x}_k) \le \phi_k^R \le \phi(\m{x}_1)$'' of Proposition
\ref{StepProposition} is satisfied at step $k$, then a choice for
$\phi_{k+1}^R$ exists which satisfies this hypothesis at step $k+1$.
\end{remark}

\begin{remark}
\label{GLL_OK}
We now show that the
GLL reference value $\phi_k^{\max}$ satisfies the
condition $\phi (\m{x}_k) \le \phi_k^R \le \phi(\m{x}_1)$ of
Proposition $\ref{StepProposition}$ for each $k$.
The condition $\phi_k^{\max} \ge \phi(\m{x}_k)$ is a trivial
consequence of the definition of $\phi_k^{\max}$.
Also, by the definition, we have $\phi_1^{\max} = \phi(\m{x}_1)$.
For $k \ge 1$, $\phi(\m{x}_{k+1}) \le \phi_k^{\max}$
according to Step~2 of SpaRSA.
Hence, $\phi_k^{\max}$ is a decreasing function of $k$.
In particular, $\phi_k^{\max} \le \phi_1^{\max} = \phi(\m{x}_1)$.
\end{remark}

\section{Convergence estimate for convex functions}
\label{convergence}
In this section we give a sublinear convergence estimate for the
error in the objective function value $\phi(\m{x}_k)$
assuming $f$ is convex and the assumptions of
Proposition \ref{StepProposition} hold.

By (A1) and (A2), (\ref{main}) has a solution $\m{x}^* \in \C{L}$ and
an associated objective function value $\phi^* := \phi(\m{x}^*)$.
The convergence of the objective function values to $\phi^*$
is a consequence of the analysis in \cite{wright09a}:
\begin{lemma}
\label{phi_k_converge}
If $\mbox{\rm (A1)}$ and $\mbox{\rm (A2)}$ hold and $\phi_k^R = \phi_k^{\max}$
for every $k$, then
\[
\lim_{k \rightarrow \infty} \phi(\m{x}_k) = \phi^*.
\]
\end{lemma}

\begin{proof}
By \cite[Lemma 4]{wright09a}, the objective function values
$\phi(\m{x}_k)$ approach a limit denoted $\bar{\phi}$.
By \cite[Theorem 1]{wright09a}, all accumulation points of the
iterates $\m{x}_k$ are stationary points.
An accumulation point exists since $\C{K}$ is compact and the
iterates are all contained in $\C{L} \subset \C{K}$,
as shown in Remark \ref{GLL_OK}.
Since $f$ and $\psi$ are both convex, a stationary
point is a global minimizer of $\phi$.
Hence, $\bar{\phi} = \phi^*$.
\end{proof}

Our sublinear convergence result is the following:
\bigskip

\begin{theorem}
\label{theorem_sublinear}
If $\mbox{\rm (A1)}$ and $\mbox{\rm (A2)}$ hold, $f$ is convex,
and $\phi_k^R = \phi_k^{\max}$ for all $k$,
then there exist constants $a$ and $b$ such that
\[
\phi (\m{x}_k) - \phi^* \le \frac{a}{b + k}
\]
for $k$ sufficiently large.
\end{theorem}
\bigskip

\begin{proof}
By (\ref{phiPhi}) with $k+1$ replaced by $k$,
we have
\begin{equation}\label{h1}
\phi(\m{x}_k) \le \Phi_{k-1} (\m{x}_k) + b_0 \|\m{s}_k\|^2, \quad
b_0 = .5\lambda,
\end{equation}
where $\m{s}_k = \m{x}_k - \m{x}_{k-1}$.
Since $\m{x}_k$ minimizes $\Phi_{k-1}$ and $f$ is convex, it follows that
\begin{eqnarray}
\Phi_{k-1}(\m{x}_k) &=& \displaystyle \min_{\m{z} \in \mathbb{R}^n}
\{ f(\m{x}_{k-1}) + \nabla f (\m{x}_{k-1})(\m{z} - \m{x}_{k-1})
+ \alpha_{k-1} \|\m{z} - \m{x}_{k-1}\|^2 + \psi(\m{z}) \} \nonumber \\
&\le& \min \{ f(\m{z}) + \psi(\m{z})
+ \alpha_{k-1} \|\m{z} - \m{x}_{k-1}\|^2 :
\m{z} \in \mathbb{R}^n \} \nonumber \\
&=& \min \{ \phi(\m{z})
+ \alpha_{k-1} \|\m{z} - \m{x}_{k-1}\|^2 :
\m{z} \in \mathbb{R}^n \} , \label{h3}
\end{eqnarray}
where $\alpha_{k-1}$ is the terminating value of $\alpha$ at step $k-1$.
Combining (\ref{h1}) and (\ref{h3}) gives
\begin{equation}\label{h4}
\phi(\m{x}_k) \le \min \{ \phi(\m{z})
+ \bar{\beta} \|\m{z} - \m{x}_{k-1}\|^2 :
\m{z} \in \mathbb{R}^n \} +
b_0 \|\m{s}_{k}\|^2 ,
\end{equation}
where $\bar{\beta} = \eta \bar{\alpha}$
is an upper bound for the $\alpha_k$ implied
by Proposition \ref{StepProposition}.
By the convexity of $\phi$ and with
$\m{z} = (1-\lambda)\m{x}_{k-1} + \lambda \m{x}^*$ for any
$\lambda \in [0, 1]$, we have
\begin{eqnarray*}
\min_{\m{z} \in \mathbb{R}^n}
\phi(\m{z}) + \bar{\beta} \|\m{z} - \m{x}_{k-1}\|^2
&\le&
\phi(
(1-\lambda)\m{x}_{k-1} + \lambda \m{x}^*) +
\bar{\beta}\lambda^2\|\m{x}_{k-1} - \m{x}^*\|^2 \\
&\le&
(1-\lambda ) \phi(\m{x}_{k-1}) + \lambda \phi^*
+ \bar{\beta}\lambda^2\|\m{x}_{k-1} - \m{x}^*\|^2 \\
&=&
(1-\lambda ) \phi(\m{x}_{k-1}) + \lambda \phi^*
+ b_k \lambda^2 ,
\end{eqnarray*}
where $b_k = \bar{\beta} \|\m{x}_{k-1} - \m{x}^*\|^2$.
Combining this with (\ref{h4}) yields
\begin{eqnarray}
\phi(\m{x}_k) &\le& (1-\lambda) \phi(\m{x}_{k-1})
+ \lambda \phi^* + b_k \lambda^2
+ b_0 \| \m{s}_k\|^2 \nonumber \\
&\le& (1-\lambda)\phi_{k-1}^R + \lambda \phi^* + b_k\lambda^2
+ b_0 \|\m{s}_k \|^2 \label{h5}
\end{eqnarray}
for any $\lambda \in [0,1]$.
Define
\begin{equation}\label{phi}
\phi_i = \max \{ \phi (\m{x}_k) : (i-1)M < k \le iM \} = \phi_{iM}^R ,
\end{equation}
and let $k_i$ denote the index $k$ where the maximum is attained.
Since $\phi(\m{x}_{k+1}) \le \phi_k^R$ in Step 2 of SpaRSA,
it follows that $\phi_k^R = \phi_k^{\max}$ is a nonincreasing function of $k$.
By (\ref{h5}) with $k = k_i$ and by the monotonicity of $\phi_k^R$,
we have
\begin{equation}\label{phi_i}
\phi_i \le (1-\lambda)\phi_{i-1} + \lambda \phi^*
+ b_{k_i} \lambda^2 + b_0 \|\m{s}_{k_i}\|^2
\end{equation}
for any $\lambda \in [0,1]$.
Since both $\m{x}_{k-1}$ and $\m{x}^*$ lie in $\C{L}$, it follows that
\begin{equation}\label{b_k}
b_k = \bar{\beta} \|\m{x}_{k-1} - \m{x}^*\|^2 \le
\bar{\beta} (\mbox{diameter of }\C{L})^2 := b_2 < \infty .
\end{equation}
Step 2 of SpaRSA implies that
\[
\|\m{s}_k\|^2 \le (\phi_{k-1}^R - \phi(\m{x}_k))/b_1
\]
where $b_1 = \sigma \alpha_{\min}$.
We take $k = k_i$ and again exploit the monotonicity of
$\phi_k^R$ to obtain
\begin{equation}\label{s_k_i}
\|\m{s}_{k_i}\|^2 \le (\phi_{i-1} - \phi_i)/b_1 .
\end{equation}
Combining (\ref{phi_i})--(\ref{s_k_i}) gives
\begin{equation}\label{phi_no_sk}
\phi_i \le (1-\lambda)\phi_{i-1} + \lambda \phi^*
+ b_2 \lambda^2 + b_3 (\phi_{i-1} - \phi_i), \quad
b_3 = b_0/b_1,
\end{equation}
for every $\lambda \in [0, 1]$,
The minimum on the right side is attained with the choice
\begin{equation}\label{lambda_min}
\lambda = \min \left\{1, \frac{\phi_{i-1} - \phi^*}{2b_2} \right\}.
\end{equation}
As a consequence of Lemma \ref{phi_k_converge},
$\phi_{i-1}$ converges to $\phi^*$.
Hence, the minimizing $\lambda$ also approaches 0 as $i$ tends to $\infty$.
Choose $k$ large enough that the minimizing $\lambda$ is less than 1.
It follows from (\ref{phi_no_sk}) that for this minimizing choice of
$\lambda$, we have
\begin{equation}\label{h6}
\phi_i \le \phi_{i-1}
- \frac{(\phi_{i-1} - \phi^*)^2}{4b_2} + b_3 (\phi_{i-1} - \phi_i) .
\end{equation}
Define $e_i = \phi_i - \phi^*$.
Subtracting $\phi^*$ from each side of (\ref{h6}) gives
\begin{eqnarray*}
e_i &\le& e_{i-1}  - e_{i-1}^2/(4b_2) + b_3 (e_{i-1} - e_i) \\
&=& (1+b_3)e_{i-1} - e_{i-1}^2/(4b_2) - b_3 e_i .
\end{eqnarray*}
We arrange this to obtain
\begin{equation}\label{h7}
e_i \le e_{i-1} - b_4 e_{i-1}^2 \quad \mbox{where }
b_4 = \frac{1}{4b_2(1+b_3)} .
\end{equation}
By (\ref{h7}) $e_i \le e_{i-1}$, which implies that
\[
e_i \le e_{i-1} - b_4 e_{i-1}e_i \quad \mbox{or} \quad
e_i \le \frac{e_{i-1}}{1+b_4 e_{i-1}} .
\]
We form the reciprocal of this last inequality to obtain
\[
\frac{1}{e_i} \ge \frac{1}{e_{i-1}} + b_4 .
\]
Applying this inequality recursively gives
\[
\frac{1}{e_i} \ge \frac{1}{e_{j}} + (i-j)b_4 \quad \mbox{or} \quad
e_i \le \frac{e_j}{1 + (i-j)b_4 e_j} ,
\]
where $j$ is chosen large enough to ensure that the
minimizing $\lambda$ in (\ref{lambda_min})
is less than 1 for all $i \ge j$.

Suppose that $k \in ((i-1)M, iM]$ with $i > j$.
Since $i \ge k/M$, we have
\[
\phi(\m{x}_k) - \phi^* \le e_i \le
\frac{e_j}{1 + (i-j)b_4 e_j} \le \frac{e_j}{1 - jb_4e_j + k b_4e_j/M} \; .
\]
The proof is completed by taking
$a = M/b_4$ and $b = M/(b_4e_j) - Mj$.
\end{proof}

\section{Convergence estimate for strongly convex functions}
\label{strong_convexity}
In this section we prove that SpaRSA converges R-linearly
when $f$ is a convex function and $\phi$ satisfies
\begin{equation}\label{StrongConvexity}
\phi (\m{y}) \ge \phi(\m{x}^*) + \mu \|\m{y} - \m{x}^*\|^2
\end{equation}
for all $\m{y} \in \mathbb{R}^n$, where $\mu > 0$.
Hence, $\m{x}^*$ is a unique minimizer of $\phi$.
For example, if $f$ is a strongly convex function, then
(\ref{StrongConvexity}) holds.
\bigskip

\begin{theorem}
\label{theorem_linear}
If $\mbox{\rm (A1)}$ and $\mbox{\rm (A2)}$ hold,
$f$ is convex, $\phi$ satisfies $(\ref{StrongConvexity})$,
and $\phi_k^R = \phi_k^{\max}$ for every $k$,
then there exist constants $\theta \in (0,1)$ and $c$ such that
\begin{equation}\label{linear_convergence}
\phi (\m{x}_k) - \phi^* \le
c\theta^k (\phi(\m{x}_1) - \phi^*)
\end{equation}
for every $k$.
\end{theorem}
\bigskip

\begin{proof}
Let $\phi_i$ be defined as in (\ref{phi}).
We will show that there exist $\gamma \in (0,1)$ such that
\begin{equation}\label{linear_bound}
\phi_i - \phi^* \le \gamma (\phi_{i-1} - \phi^*) .
\end{equation}
Let  $c_1$ be chosen to satisfy the inequality
\begin{equation}\label{c1}
0 < c_1 < \min \left\{ \frac{1}{2b_0}, \frac{\mu}{4b_0 \bar{\beta}} \right\} .
\end{equation}
We consider 2 cases.

{\it Case} 1. $\|\m{s}_{k_i}\|^2 \ge c_1 (\phi_{i-1} - \phi^*)$.

\noindent
By (\ref{s_k_i}), we have
\[
c_1 (\phi_{i-1} - \phi^*) \le (\phi_{i-1} - \phi_i)/b_1 .
\]
This can be rearranged to obtain
\[
\phi_i - \phi^* \le (1-b_1 c_1) (\phi_{i-1} - \phi^*) ,
\]
which yields (\ref{linear_bound}).

{\it Case} 2. $\|\m{s}_{k_i}\|^2 < c_1 (\phi_{i-1} - \phi^*)$.

\noindent
We utilize the inequality (\ref{phi_i}) but with different bounds for
the $b_{k_i}$ and $\m{s}_{k_i}$ terms.
For $k \in ((i-1)M, iM]$, we have
\begin{eqnarray}
b_k := \bar{\beta} \|\m{x}_{k-1} - \m{x}^*\|^2
&\le& \frac{\bar{\beta}}{\mu} (\phi(\m{x}_{k-1}) - \phi^*)
\le \frac{\bar{\beta}}{\mu} (\phi_{k-1}^R - \phi^*) \nonumber \\
&\le& \frac{\bar{\beta}}{\mu} (\phi_{(i-1)M}^R - \phi^*)
= b_5 (\phi_{i-1} - \phi^*) , \quad b_5 = \frac{\bar{\beta}}{\mu}.
\nonumber
\end{eqnarray}
The first inequality is due to (\ref{StrongConvexity})
and the last inequality is since $\phi_k^R$ is monotone decreasing.
By the definition of $k_i$ below (\ref{phi}), it follows that
$k_i \in ((i-1)M, iM]$ and
\begin{equation}\label{bk}
b_{k_i} \le b_5 (\phi_{i-1} - \phi^*) .
\end{equation}
Inserting in (\ref{phi_i}) the bound (\ref{bk}) and the Case 2 requirement
$\|\m{s}_{k_i}\|^2 < c_1 (\phi_{i-1} - \phi^*)$ yields
\[
\phi_i \le (1-\lambda)\phi_{i-1} + \lambda \phi^*
+ b_5 (\phi_{i-1} - \phi^*) \lambda^2 +
b_0c_1 (\phi_{i-1} - \phi^*)
\]
for all $\lambda \in [0,1]$.
Subtract $\phi^*$ from each side to obtain
\begin{equation}\label{e_i}
e_i \le [1+b_0c_1 - \lambda + b_5 \lambda^2] e_{i-1}
\end{equation}
for all $\lambda \in [0,1]$.

The $\lambda \in [0,1]$ which minimizes the coefficient of $e_{i-1}$
in (\ref{e_i}) is
\[
\lambda = \min \left\{ 1, \frac{1}{2b_5} \right \} .
\]
If the minimizing $\lambda$ is 1, then $b_5 \le 1/2$ and the
minimizing coefficient in (\ref{e_i}) is
\[
\gamma = b_0 c_1 + b_5 \le b_0c_1 + 1/2 < 1
\]
since $c_1 < 1/(2b_0)$ by (\ref{c1}).
On the other hand, if the minimizing $\lambda$ is less than 1,
then $b_5 > 1/2$ and the minimizing coefficient is
\[
\gamma = 1 + b_0 c_1 - \frac{1}{4b_5}  < 1
\]
since $1/(4 b_5) = \mu/(4 \bar{\beta}) > b_0 c_1$ by (\ref{c1}).
This completes the proof of (\ref{linear_bound}).

For $k \in ((i-1)M, iM]$, we have
\[
\phi(\m{x}_k) - \phi^* \le e_i \le \gamma^{i-1} e_1 \le
\frac{1}{\gamma} \left( \gamma^{1/M} \right)^k (\phi(\m{x}_1) - \phi^*) .
\]
Hence, (\ref{linear_convergence}) holds with $c = 1/\gamma$ and
$\theta = \gamma^{1/M}$.
This completes the proof.
\end{proof}
\smallskip

\begin{remark}
The condition $(\ref{StrongConvexity})$ when combined with
(\ref{linear_convergence}) shows that the iterates
$\m{x}_k$ converge R-linearly to $\m{x}^*$.
\end{remark}

\section{More general reference function values}
\label{MSparsa}

The GLL reference function value $\phi_k^{\max}$,
defined in (\ref{GLL}),
often leads to greater efficiency when $M > 1$, when compared to
the monotone choice $M = 1$.
In practice, it is found that even more flexibility
in the reference function value can further accelerate convergence.
In \cite{hz05a} we prove convergence
of the nonmonotone gradient projection method whenever the
reference function $\phi_k^R$ satisfies the following conditions:
\smallskip

\begin{itemize}
\item[(R1)]
$\phi_1^R = \phi(\m{x}_1)$.
\item[(R2)]
$\phi (\m{x}_k) \le \phi_k^R \le \max\{\phi_{k-1}^R, \phi_k^{\max}\}$
for each $k > 1$.
\item[(R3)]
$\phi_k^R \le \phi_k^{\max}$ infinitely often.
\end{itemize}
\smallskip

In \cite{hz05a} we provide a specific choice for $\phi_k^R$
which satisfies (R1)--(R3) and which gave more rapid convergence
than the choice $\phi_k^R = \phi_k^{\max}$.
To satisfy (R3), we could choose an integer $L > 0$ and simply
set $\phi_k^R = \phi_k^{\max}$ every $L$ iterations.
Another strategy, closer in spirit to what is used in the
numerical experiments, is to choose a decrease parameter $\Delta > 0$
and set $\phi_k^R = \phi_k^{\max}$ if
$\phi(\m{x}_{k-L}) - \phi(\m{x}_k) \le \Delta$.
We now give convergence results for SpaRSA
whenever the reference function value satisfies (R1)--(R3).
In the first convergence result which follows,
convexity of $f$ is not required.
\smallskip

\begin{theorem} \label{liminf}
If $\mbox{\rm (A1)}$ and $\mbox{\rm (A2)}$ hold
and the reference function value $\phi_k^R$ satisfies
{\rm (R1)}--{\rm (R3)},
then the iterates $\m{x}_k$ of SpaRSA have a subsequence converging
to a limit $\bar{\m{x}}$ satisfying $\m{0} \in \partial \phi(\bar{\m{x}})$.
\end{theorem}
\smallskip

\begin{proof}
We first apply Proposition \ref{StepProposition}
to show that Step~2 of SpaRSA is fulfilled for some choice of $j$.
This requires that we show $\phi_k^R \le \phi(\m{x}_1)$ for each $k$.
This holds for $k = 1$ by (R1).
Also, for $k = 1$, we have $\phi_1^{\max} = \phi(\m{x}_1)$.
Proceeding by induction, suppose that
$\phi_i^R \le \phi(\m{x}_1)$ and $\phi_i^{\max} \le \phi (\m{x}_1)$
for $i = 1$, 2, $\ldots$, $k$.
By Proposition \ref{StepProposition}, Step 2 of SpaRSA terminates
at a finite $j$ and hence,
\[
\phi(\m{x}_{k+1}) \le \phi_k^R \le \phi (\m{x}_1) .
\]
It follows that $\phi_{k+1}^{\max} \le \phi(\m{x}_1)$ and
$\phi_{k+1}^R \le \max \{\phi_k^R, \phi_{k+1}^{\max} \} \le$
$\phi(\m{x}_1)$.
This completes the induction step, and hence,
by Proposition \ref{StepProposition}, it follows that in every
iteration, Step~2 of SpaRSA is fulfilled for a finite $j$.

By Step 2 of SpaRSA, we have
\[
\phi(\m{x}_k) \le \phi_{k-1}^R - \sigma \alpha_{\min}\|\m{s}_{k}\|^2 ,
\]
where $\m{s}_{k} = \m{x}_k - \m{x}_{k-1}$.
In the third paragraph of the proof of Theorem 2.2 in \cite{hz05a},
it is shown that when an inequality of this form is satisfied for
a reference function value satisfying (R1)--(R3), then
\[
\lim \inf_{k\rightarrow \infty} \|\m{s}_k\| = 0 .
\]

Let $k_i$ denote a strictly increasing sequence with the property that
$\m{s}_{k_i}$ tends to $\m{0}$ and $\m{x}_{k_i}$ approaches
a limit denoted $\bar{\m{x}}$.
That is,
\[
\lim_{i\rightarrow \infty} \m{s}_{k_i} = 0  \quad \mbox{and} \quad
\lim_{i\rightarrow \infty} \m{x}_{k_i} = \bar{\m{x}} .
\]
Since $\m{s}_{k_i}$ tends to $\m{0}$, it follows that
$\m{x}_{k_i - 1}$ also approaches $\bar{\m{x}}$.
By the first-order optimality conditions for $\m{x}_{k_i}$, we have
\begin{equation}\label{optimality}
\m{0} \in \nabla f(\m{x}_{k_i - 1}) +
2\alpha_{k_i}(\m{x}_{k_i} - \m{x}_{k_i-1}) +
\partial \psi (\m{x}_{k_i}) ,
\end{equation}
where $\alpha_{k_i}$ denotes the value of $\alpha$ in Step 2
of SpaRSA associated with $\m{x}_{k_i}$.
Again, by Proposition \ref{StepProposition},
we have the uniform bound
$\alpha_{k_i} \le \bar{\beta} = \eta \bar{\alpha}$.
Taking the limit as $i$ tends to $\infty$,
it follows from Corollary 24.5.1 in \cite{Rockafellar70} that
\[
\m{0} \in \nabla f(\bar{\m{x}}) + \partial \psi (\bar{\m{x}}).
\]
This completes the proof.
\end{proof}

With a small change in (R3), we obtain either sublinear or linear
convergence of the entire iteration sequence.

\begin{theorem} \label{linear_R}
Suppose that $\mbox{\rm (A1)}$ and $\mbox{\rm (A2)}$ hold, $f$ is convex,
the reference function value $\phi_k^R$ satisfies
{\rm (R1)} and {\rm (R2)}, and there is $L > 0$ with the
property that for each $k$,
\begin{equation}\label{kL}
\phi_j^R \le \phi_j^{\max} \quad \mbox{for some } j \in [k, k+L) .
\end{equation}
Then there exist constants $a$ and $b$ such that
\[
\phi (\m{x}_k) - \phi^* \le \frac{a}{b + k}
\]
for $k$ sufficiently large.
Moreover, if $\phi$ satisfies the strong convexity condition
$(\ref{StrongConvexity})$,
then there exists $\theta \in (0,1)$ and $c$ such that
\[
\phi (\m{x}_k) - \phi^* \le
c\theta^k (\phi(\m{x}_1) - \phi^*)
\]
for every $k$.
\end{theorem}
\smallskip

\begin{proof}
Let $k_i$, $i = 1, 2, \ldots$,
denote an increasing sequence of integers
with the property that $\phi_j^R \le \phi_j^{\max}$ for $j = k_i$ and
$\phi_j^R \le \phi_{j-1}^R$ when $k_i < j < k_{i+1}$.
Such a sequence exists since
$\phi_k^R \le \max\{\phi_{k-1}^R, \phi_k^{\max}\}$ for each $k$ and
(\ref{kL}) holds.
Moreover, $k_{i+1} - k_i \le L$.
Hence, we have
\begin{equation}\label{Rmax}
\phi_j^R \le \phi_{k_i}^R \le \phi_{k_i}^{\max}, \quad
\mbox{when } k_i \le j < k_{i+1} .
\end{equation}
Let us define
\[
\phi_j^{\max +} =
\max \{ \phi (\m{x}_{j-i} : 0 \le i < \min (j, M+L) \} .
\]
Given $j$, choose $k_i$ such that $j \in [k_i, k_{i+1})$.
Since $j - k_i < L$, the set of function values maximized to obtain
$\phi_{k_i}^{\max}$ is contained in the set of
function values maximized to obtain $\phi_j^{\max+}$ and
we have
\begin{equation}\label{z2}
\phi_{k_i}^{\max} \le \phi_j^{\max+} .
\end{equation}
Combining (\ref{Rmax}) and (\ref{z2}) yields
$\phi_j^R \le \phi_j^{\max +}$ for each $j$.
In Step 2 of SpaRSA, the iterates are chosen to satisfy the
condition
\[
\phi (\m{x}_{k+1}) \le \phi_k^R - \sigma \alpha
\|\m{x}_{k+1} - \m{x}_k\|^2 .
\]
It follows that
\[
\phi (\m{x}_{k+1}) \le \phi_k^{\max +} - \sigma \alpha
\|\m{x}_{k+1} - \m{x}_k\|^2 .
\]
Hence, the iterates also satisfy the GLL condition, but with
memory of length $M+L$ instead of $M$.
By Theorem $\ref{theorem_sublinear}$, the iterates converge
at least sublinearly.
Moreover, if the strong convexity condition
$(\ref{StrongConvexity})$ holds, then the convergence is
R-linear by Theorem \ref{theorem_linear}.
\end{proof}

\section{Computational experiments}
\label{experiments}
In this section, we compare the performance of SpaRSA with the GLL
reference function value $\phi_k^{\max}$ and the BB choice for $\alpha_0$
in SpaRSA, to that of an adaptive implementation based on the
reference function value $\phi_k^R$ given in the appendix of \cite{hz05a}
and a cyclic BB choice for $\alpha_0$.
We call this implementation Adaptive SpaRSA.
This adaptive choice for $\phi_k^R$ satisfies (R1)--(R3) which ensures
convergence in accordance with Theorem \ref{liminf}.
By a cyclic choice for the BB parameter
(see \cite{dyh01, dhsz05, fmmr99, rs02}),
we mean that $\alpha_0 = \alpha_k^{BB}$ is reused for several iterations.
More precisely, for some integer $m \ge 1$ (the cycle length),
and for all $k \in ((i-1)m, im]$,
the value of $\alpha_0$ at iteration $k$ is given by
\[
(\alpha_0)_k = \alpha_{(i-1)m +1}^{BB} .
\]

The test problems are associated with applications in the areas of
signal processing and image reconstruction.
All experiments were
carried out on a PC using Matlab 7.6 with a AMD Athlon 64 X2 dual
core 3 Ghz processor and 3GB of memory running Windows Vista.
Version 2.0 of SpaSRA was obtained from M\'{a}rio Figueiredo's webpage
(http://www.lx.it.pt/$\sim$mtf/SpaRSA/).
The code was run with default parameters.
Adaptive SpaRSA was written in Matlab
with the following parameter values
\[
\alpha_{\min} = 10^{-30}, \quad \alpha_{\max} = 10^{30}, \quad \eta = 5,
\quad \sigma = 10^{-4}, \quad M = 10 .
\]
The test problems, such as the basis pursuit denoising problem
(\ref{BPDN}), involve a parameter $\tau$.
The choice of the cycle length was based on the value of $\tau$:
\[
m = 1 \mbox{ if } \tau \ge 10^{-2}, \mbox{ otherwise } m = 3 .
\]
As $\tau$ approaches zero, the optimization problem becomes
more ill conditioned and the convergence speed improves when
the cycle length is increased.

The stopping condition for both SpaRSA and Adaptive SpaRSA was
\[
\alpha_k \| \m{x}_{k+1} - \m{x}_k \|_{\infty} \le \epsilon ,
\]
where $\alpha_k$ denotes the final value for $\alpha$ in Step 2 of
SpaRSA, $\| \cdot \|_{\infty}$ is the max-norm, and
$\epsilon$ is the error tolerance.
This termination condition is suggested by Vandenberghe in
\cite{Vandenberghe09}.
As pointed out earlier, $\m{x}_k$ is a stationary point when
$\m{x}_{k+1} = \m{x}_k$.
For other stopping criteria, see \cite{HaleZhang07} or \cite{wright09a}.
In the following tables, ``Ax" denotes the number of times that a
vector is multiplied by $\m{A}$ or $\m{A}\tr$,
``cpu" is the CPU time in seconds, and ``Obj" is the objective function value.



\subsection{$\ell_2-\ell_1$ problems}
\label{l2l1}

We compare the performance of Adaptive SpaRSA with SpaRSA
by solving $\ell_2-\ell_1$ problems of
form (\ref{BPDN}) using the randomly generated data introduced in
\cite{BoydL107,wright09a}. The matrix $\m{A}$ is a random
$k\times n$ matrix, with $k = 2^8$ and $n=2^{10}$.
The elements of $\m{A}$ are chosen from a
Gaussian distribution with mean zero and variance $1/(2n)$.
The observed vector is $\m{b} = \m{Ax}_{true} + \m{n}$,
where the noise $\m{n}$ is sampled from
a Gaussian distribution with mean zero and variance $10^{-4}$.
$\m{x}_{true}$ is a vector with 160 randomly placed $\pm1$ spikes
with zeros in the remaining elements. This is a typical sparse
signal recovery problem which often arises in compressed sensing
\cite{Wright07l1}. We solved the problem (\ref{BPDN}) corresponding
to the error tolerance $10^{-5}$ with different regularization
parameters $\tau$ between $10^{-1}$ and $10^{-5}$.
Table \ref{tab1} reports the
average cpu times (seconds) and the number of matrix-vector
multiplications over 10 runs for both the original SpaRSA algorithm
and an implementation based on a continuation method (see \cite{HaleZhang07}).
The implementations using the continuation method are indicated by ``/c''
in Table \ref{tab1}.
These results show
that the Adaptive SpaRSA is significantly faster than SpaSRA
when not using the continuation technique.
The performance gap
decreases when the continuation technique is applied.
Nonetheless, Adaptive SpaRSA yields better performance.

Figure \ref{Err} plots error versus the
number of matrix-vector multiplication
for $\tau=10^{-4}$ and the implementation without continuation.
When the error is large, both algorithm have the same performance.
As the error tolerance decreases, the performance of the adaptive
algorithm is significantly better than the original implementation.


\def\baselinestretch{1}
\begin{table}[h!b!p!]
\caption{Average over 10 runs for $\ell_2-\ell_1$ problems}
\begin{small}
\begin{center}
\begin{tabular}{|c|rr|rr|rr|rr|rr|}
    \hline
  $\tau$        &  & 1e-1&  & 1e-2 &  & 1e-3  & & 1e-4  & &  1e-5  \\
    \hline
   &Ax &cpu& Ax &cpu &Ax &cpu &Ax&cpu&Ax &cpu \\
\hline
SpaRSA   & 65.3   &.07 & 706.4  & .56 & 3467.5  & 2.73  & 8802.9  & 6.86& 5925.5 & 4.65 \\
Adaptive & 65.4   &.07 & 582.8  & .44 & 1998.8  & 1.58  & 4394.0  & 3.50& 2911.9 & 2.36 \\
SpaRSA/c   & 65.3   &.07 & 626.7  & .48 & 2172.1  & 1.67  & 684.9   & .52 & 474.8  & .36  \\
Adaptive/c   & 65.4   &.07 & 569.0  & .44 & 1928.3  & 1.51  & 636.0   & .50 & 453.7  & .34  \\
\hline
\end{tabular}
\label{tab1}
\end{center}
\end{small}
\end{table}
\def\baselinestretch{1}

\begin{figure}
\begin{center}
\includegraphics[width=14cm,clip]{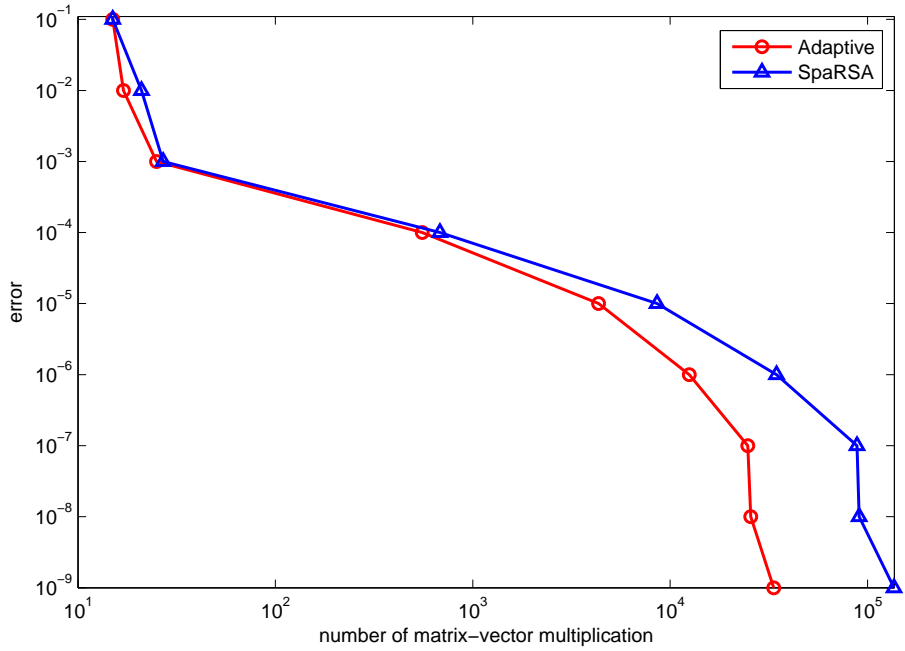}
\end{center}
\caption[Short caption for figure 1]{\label{Err} Number of matrix-vector
multiplications versus error}
\end{figure}

\subsection{Image deblurring problems}
In this subsection, we present
results for two image restoration
problems based on images referred to as
\emph{Resolution} and \emph{Cameraman}.
The images are $256\times256$ gray scale images; that is,
$n = 256^2 = 65536$.
The images are blurred by convolution with an $8\times8$
blurring mask and normally distributed noise with standard deviation
$0.0055$ is added to the final signal
(see problem 701 in \cite{Berg09}).
The image restoration problem has the form (\ref{BPDN})
where $\tau = 0.00005$ and
%
%
$\m{A} = \m{HW}$ is the composition of the
blur matrix and the Haar discrete wavelet transform (DWT) operator.
For these test problems, the continuation approach is no faster,
and in some cases significantly slower,
than the implementation without continuation.
Therefore, we solved these test problems without the
continuation technique.
The results in Table \ref{tab2} again indicate that the adaptive scheme
yields much better performance as the error tolerance decreases.

\begin{figure}
\begin{minipage}[t]{8cm}
\begin{center}
\includegraphics[width=7.7cm,clip]{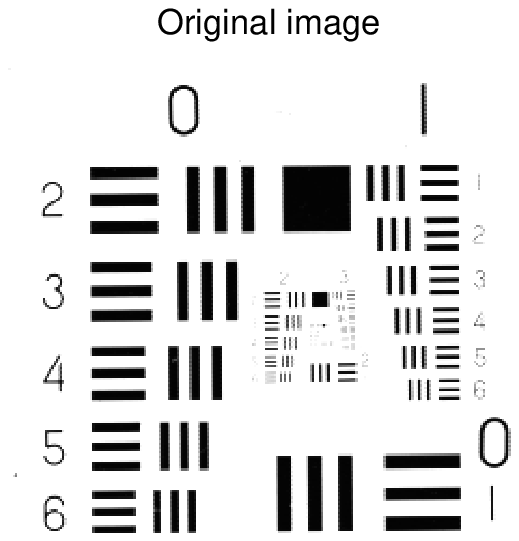}
\end{center}
\vspace{-1.3cm}
\begin{center}
\includegraphics[width=7.7cm,clip]{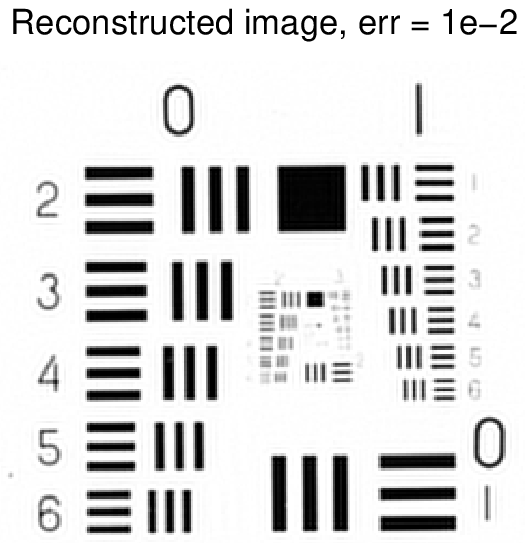}
\end{center}
\end{minipage}
\hfill \hspace{-1.9cm} \begin{minipage}[t]{8cm}
\begin{center}
\includegraphics[width=7.7cm,clip]{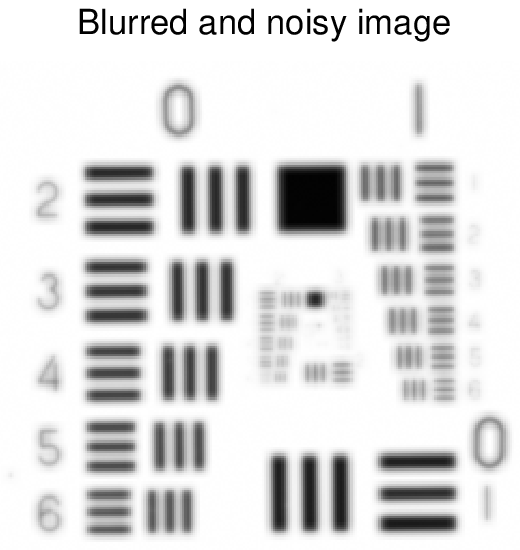}
\end{center}
\vspace{-1.3cm}
\begin{center}
\includegraphics[width=7.7cm,clip]{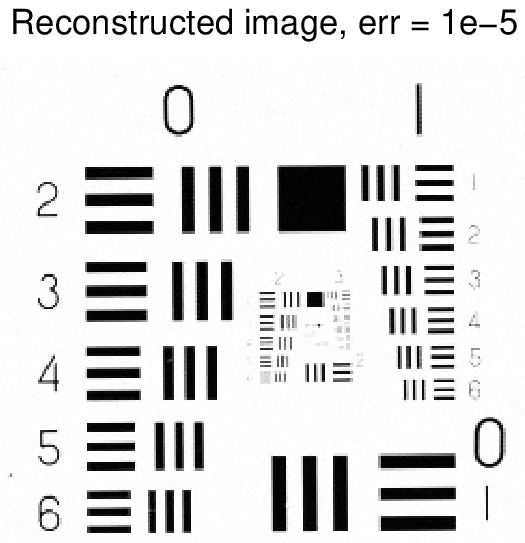}
\end{center}
\end{minipage}
\vspace{-1.2cm} \caption[Short caption for figure 1]{\label{resol}
Deblurring the resolution image}
\end{figure}
\begin{figure}
\begin{minipage}[t]{8cm}
\begin{center}
\includegraphics[width=7.8cm,clip]{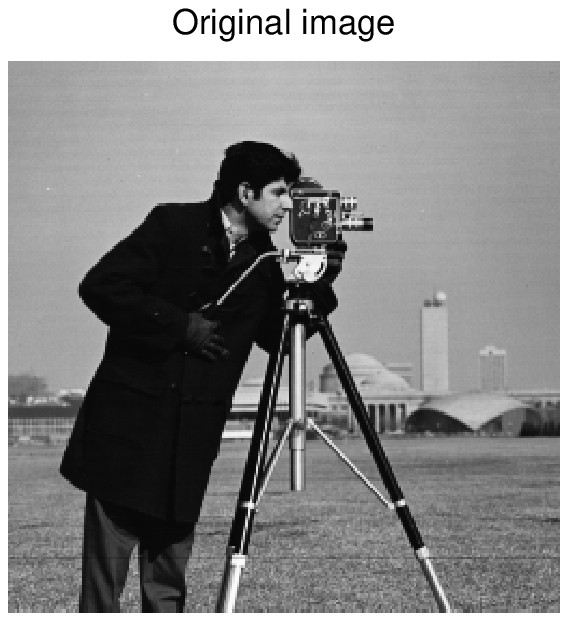}
\end{center}
\vspace{-0.8cm}
\begin{center}
\includegraphics[width=7.8cm,clip]{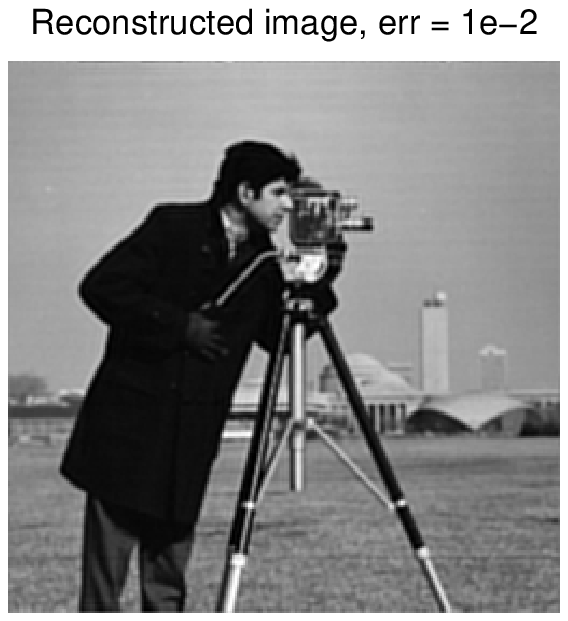}
\end{center}
\end{minipage}
\hfill \hspace{-2.0cm} \begin{minipage}[t]{8cm}
\begin{center}
\includegraphics[width=7.8cm,clip]{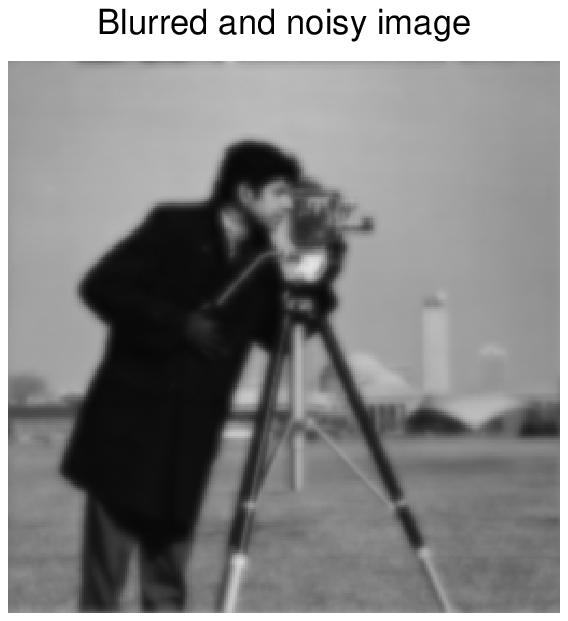}
\end{center}
\vspace{-0.8cm}
\begin{center}
\includegraphics[width=7.8cm,clip]{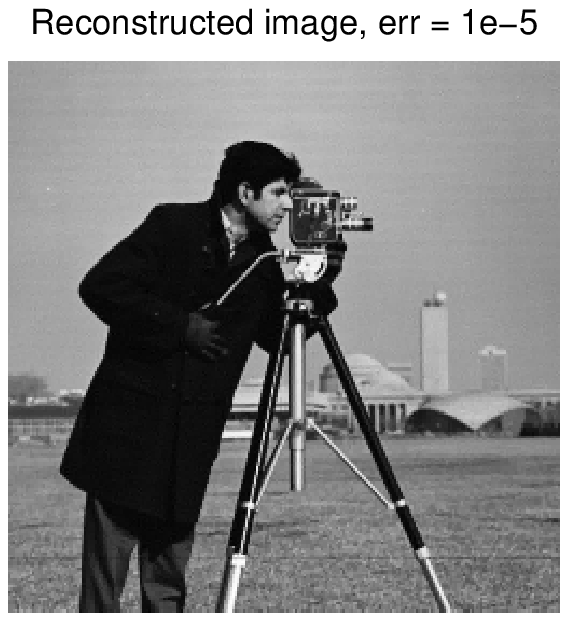}
\end{center}
\end{minipage}
\vspace{-0.5cm} \caption[Short caption for figure
1]{\label{cameraman} Deblurring the cameraman image}
\end{figure}

\def\baselinestretch{1.3}
\begin{table}[h!b!p!]
\caption{Deblurring images}
\begin{scriptsize}
\begin{center}
\begin{tabular}{|c|rcr|rcr|rcr|rcr|}
    \hline
error          &  && 1e-2 &  && 1e-3  & & & 1e-4 & &&  1e-5  \\
    \hline
   &Ax &cpu& Obj &Ax &cpu & Obj &Ax &cpu & Obj &Ax&cpu & Obj\\

\hline
\multicolumn{13}{|l|}{Resolution} \\
\hline
SpaRSA  & 49   &2.57 & .4843  & 88 & 4.80 & .3525  & 458 &24.74& .2992 & 1679 &88.27& .2970 \\
Adaptive  & 37   &1.93 & .5619  & 73  & 4.02  & .3790  & 316 &17.28& .2981 & 681 &35.90& .2970\\
\hline

\multicolumn{13}{|l|}{Cameraman} \\
\hline
SpaRSA    & 34   &1.66 & .3491   & 77 & 3.99 & .2181  & 332 &17.08& .1880 & 1356 &69.45& .1868 \\
Adaptive      & 35    &1.71 & .3380  & 63  & 3.31  & .2232  & 215 &11.20& .1880 & 599 &31.4& .1868\\
\hline

\end{tabular}
\label{tab2}
\end{center}
\end{scriptsize}
\end{table}
\def\baselinestretch{1}

\subsection{Group-separable regularizer}

In this subsection, we examine performance
using the group separable regularizers \cite{wright09a} for which
\[
\psi(\m{x}) = \tau \sum_{i=1}^n \|\m{x}_{[i]}\|_2 ,
\]
where $\m{x}_{[1]}, \m{x}_{[2]},\ldots,\m{x}_{[m]}$ are $m$ disjoint
subvectors of $\m{x}$. The smooth part of $\phi$ can be expressed as $f(\m{x})
= \frac{1}{2}\| \m{A}\m{x} - \m{b}\|^2$, where $\m{A} \in
\mathbb{R}^{1024\times 4096}$ was obtained by orthonormalizing
the rows of a matrix constructed in Subsection \ref{l2l1}. The true vector
$\m{x}_{true}$ has 4096 components divided into $m = 64$ groups of
length $l_i = 64$. $\m{x}_{true}$ is generated by randomly choosing
8 groups and filling them with numbers chosen from a
Gaussian distribution with zero mean and
unit variance, while all other groups are filled with zeros.
The target vector is $\m{b} = \m{Ax}_{true} + \m{n}$, where $\m{n}$
is Gaussian noise with mean zero and variance $10^{-4}$.
The regularization parameter is chosen as suggested in \cite{wright09a}:
$\tau = 0.3 \|\m{A} \tr \m{b}\|_{\infty}$.
We ran 10 test problems with
error tolerance $= 10^{-5}$ and compute the average results.
Adaptive SpaRSA solved the test problem in 0.8420 seconds with
67.4 matrix/vector multiplications, while the SpaRSA obtained
similar performance: 0.8783 seconds and
69.1 matrix/vector multiplications.
Figure \ref{groupl2} shows the result
obtained by both methods for one sample.

\begin{figure}
\begin{center}
\includegraphics[width=12cm,clip]{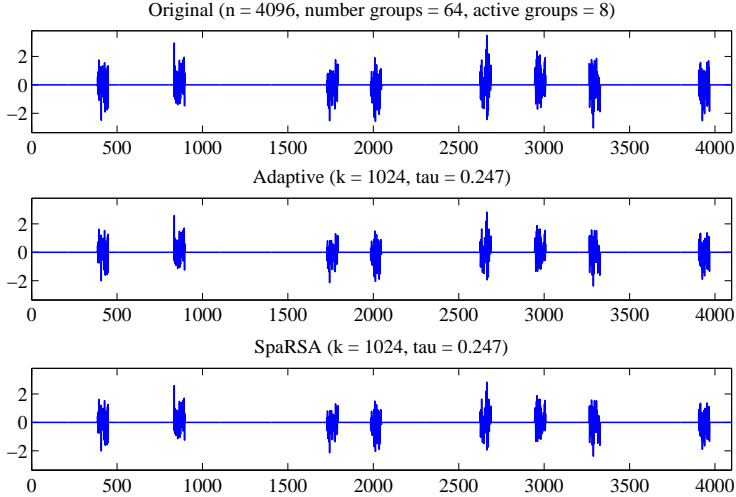}
\end{center}
\caption[Short caption for figure 1]{\label{groupl2} Group-separable
reconstruction }
\end{figure}

\subsection{Total-variation phantom reconstruction}

In this experiment, the image is the Shepp-Logan phantom of size
$256 \times 256$ (see \cite{Figueiredo07,Cand05}). The objective
function was
\[
\phi(\m{x}) = \frac{1}{2}\| \m{A}(\m{x}) - \m{b}\|^2 +
.01 \mbox{TV}(\m{x})
\]
where $\m{A}$ is a $6136\times256^2$ matrix
corresponding to 6136 locations in the 2D Fourier plane
(\verb"masked_FFT" in Matlab).
The total variation (TV) regularization is defined as follows
\[
\mbox{TV}(\m{x}) = \sum_i \sqrt{\left(\triangle_i^h \m{x}\right)^2 +
\left(\triangle_i^v \m{x}\right)^2}
\]
where $\triangle_i^h$ and $\triangle_i^v$ are linear operators
corresponding to horizontal and vertical first order differences
(see \cite{Figueiredo06}).
As seen in Table \ref{tab3},
Adaptive SpaRSA was faster than
the original SpaRSA when the error tolerance was sufficiently small.

\def\baselinestretch{1}
\begin{table}[h!b!p!]
\caption{Total-variation phantom reconstruction}
\begin{small}
\begin{center}
\begin{tabular}{|c|rcr|rcr|rcr|}
    \hline
error          &  && 1e-2 &  && 1e-3  & & & 1e-4  \\
    \hline
   &Ax &cpu& Obj &Ax &cpu & Obj &Ax &cpu & Obj \\
   \hline
SpaRSA   &14 &2.55& 36.7311 &143 &30.06 & 14.7457 &2877 &938.25 & 14.1433 \\
Adaptive &14 &2.57& 36.7311 &136 &27.32 & 14.6840 &731  &185.62 & 14.1730 \\
\hline
\end{tabular}
\label{tab3}
\end{center}
\end{small}
\end{table}
\def\baselinestretch{1}

\noindent
\begin{figure}
\begin{minipage}[t]{6.8cm}
\begin{center}
\includegraphics[width=5.8cm,clip]{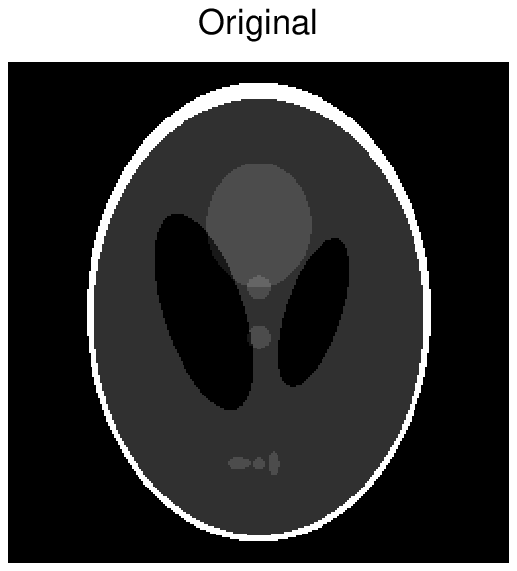}
\includegraphics[width=5.8cm,clip]{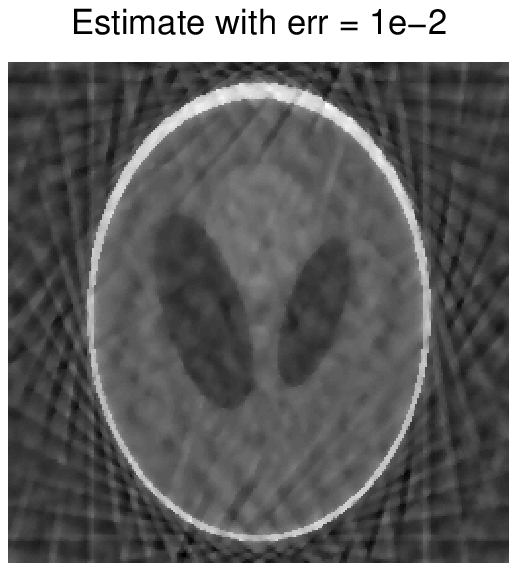}
\end{center}
\end{minipage}
\hfill \hspace{-1.2cm} \begin{minipage}[t]{6.5cm}
\begin{center}
\includegraphics[width=5.8cm,clip]{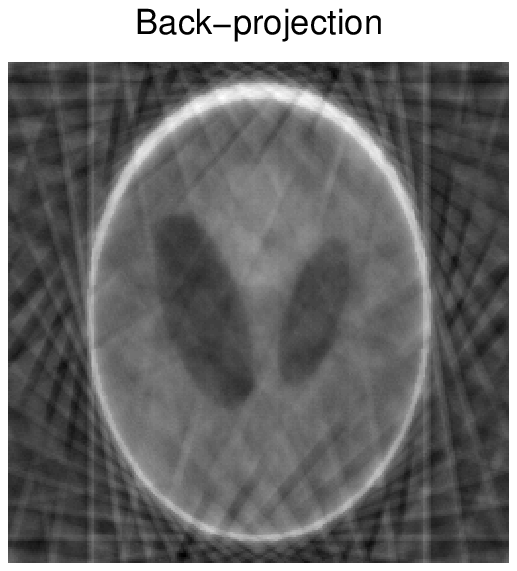}
\includegraphics[width=5.8cm,clip]{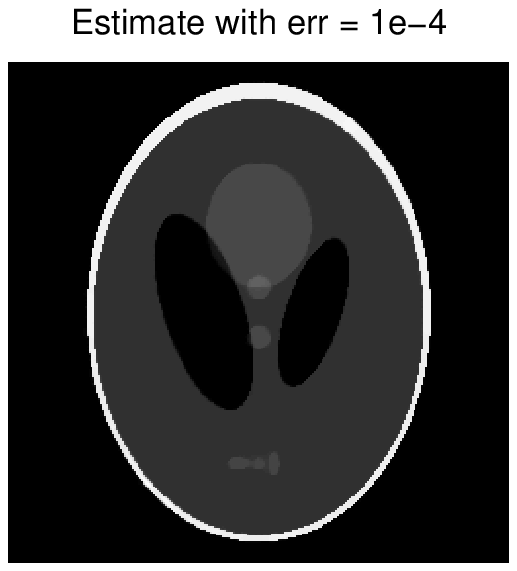}
\end{center}
\end{minipage}
\caption[Short caption for figure 1]{\label{phantom} Phantom
reconstruction}
\end{figure}

\section{Conclusions}
\label{conclusions}
The convergence properties of the
SpaRSA algorithm (Sparse Reconstruction by Separable Approximation)
of Wright, Nowak, and Figueiredo \cite{wright09a} are analyzed.
We establish sublinear convergence when
$\phi$ is convex and the GLL reference function value \cite{gll86} is employed.
When $\phi$ is strongly convex, the convergence is R-linear.
For a reference function value which satisfies (R1)--(R3),
we prove the existence of a convergent subsequence of iterates
that approaches a stationary point.
For a slightly stronger version of (R3), given in
(\ref{kL}), we show that
sublinear or linear convergence again hold when
$\phi$ is convex or strongly convex respectively.
In a series of numerical experiments, it is shown that an
Adaptive SpaRSA, based on a relaxed choice of the reference function
value and a cyclic BB iteration \cite{dhsz05, hz05a},
often yields much faster convergence,
especially when the error tolerance is small.


\begin{thebibliography}{10}

\bibitem{bb88}
{\sc J.~Barzilai and J.~M. Borwein}, {\em Two point step size gradient
  methods}, {IMA} J. Numer. Anal., 8 (1988), pp.~141--148.

\bibitem{beck09}
{\sc A.~Beck and M.~Teboulle}, {\em A fast iterative shrinkage-thresholding
  algorithm for linear inverse problems}, SIAM Journal on Imaging Sciences, 2
  (2009), pp.~183--202.

\bibitem{Figueiredo07}
{\sc J.~Bioucas-Dias and M.~Figueiredo}, {\em Twist: Two-step iterative
  shrinkage/thresholding algorithm for linear inverse problems}.
\newblock http://www.lx.it.pt/$\sim$bioucas/TwIST/TwIST.htm.

\bibitem{Figueiredo06}
{\sc J.~Bioucas-Dias, M.~Figueiredo, and J.~P. Oliveira}, {\em Total
  variation-based image deconvolution: a majorization-minimization approach.},
  in Proceedings of the IEEE International Conference on Acoustics, Speech and
  Signal Processing, vol.~2, 2006, pp.~861--864.

\bibitem{Cand05}
{\sc E.~J. Cand\`{e}s and J.~Romberg}, {\em Practical signal recovery from
  random projections.}, Wavelet Applications in Signal and Image Processing XI,
  Proc. SPIE Conf., 5914 (2005).

\bibitem{Chambolle98}
{\sc A.~Chambolle, R.~A. DeVore, N.~Y. Lee, and B.~J. Lucier}, {\em Nonlinear
  wavelet image processing: Variational problems, compression, and noise
  removal through wavelet shrinkage}, IEEE Trans. Image Process., 7 (1998),
  p.~319–335.

\bibitem{chen98}
{\sc S.~Chen, D.~Donoho, and M.~Saunders}, {\em Atomic decomposition by basis
  pursuit}, {SIAM} J. Sci. Comput., 20 (1998), pp.~33--61.

\bibitem{dyh01}
{\sc Y.~H. Dai}, {\em Alternate stepsize gradient method}, Optimization, 52
  (2003), pp.~395--415.

\bibitem{dhsz05}
{\sc Y.~H. Dai, W.~W. Hager, K.~Schittkowski, and H.~Zhang}, {\em The cyclic
  {B}arzilai-{B}orwein method for unconstrained optimization}, {IMA} J. Numer.
  Anal., 26 (2006), pp.~604--627.

\bibitem{Daubechies04}
{\sc I.~Daubechies, M.~Defrise, and C.~D. Mol}, {\em An iterative thresholding
  algorithm for linear problems with a sparsity constraint}, Comm. Pure Appl.
  Math., 57 (2004), pp.~1413--1457.

\bibitem{Wright07l1}
{\sc M.~A.~T. Figueiredo, R.~D. Nowak, and S.~J. Wright}, {\em Gradient
  projection for sparse reconstruction: Application to compressed sensing and
  other inverse problems}, IEEE Journal on Selected Topics in Signal
  Processing, 1 (2007), pp.~586--597.

\bibitem{Figueiredo03}
{\sc T.~Figueiredo and R.~D. Nowak}, {\em An {EM} algorithm for wavelet-based
  image restoration}, IEEE Trans. Image Process., 12 (2003), p.~906–916.

\bibitem{fmmr99}
{\sc A.~Friedlander, J.~M. Mart\'{\i}nez, B.~Molina, and M.~Raydan}, {\em
  Gradient method with retards and generalizations}, {SIAM} J. Numer. Anal., 36
  (1999), pp.~275--289.

\bibitem{gll86}
{\sc L.~Grippo, F.~Lampariello, and S.~Lucidi}, {\em A nonmonotone line search
  technique for {Newton}'s method}, {SIAM} J. Numer. Anal., 23 (1986),
  pp.~707--716.

\bibitem{hz05a}
{\sc W.~W. Hager and H.~Zhang}, {\em A new active set algorithm for box
  constrained optimization}, {SIAM} J. Optim., 17 (2006), pp.~526--557.

\bibitem{HaleZhang07}
{\sc E.~Hale, W.~Yin, and Y.~Zhang}, {\em A fixed-point continuation method for
  $\ell_1$-regularized minimization with applications to compressed sensing},
  tech. report, Rice University, July 2007.

\bibitem{BoydL107}
{\sc S.-J. Kim, K.~Koh, M.~Lustig, S.~Boyd, and D.~Gorinevsky}, {\em An
  interior-point method for large-scale $\ell_1$-regularized least squares},
  IEEE Journal on Selected Topics in Signal Processing, 1 (2007), pp.~606--617.

\bibitem{nesterov07}
{\sc Y.~Nesterov}, {\em Gradient methods for minimizing composite objective
  function}, CORE Discussion Papers 2007/76, Université catholique de Louvain,
  Center for Operations Research and Econometrics (CORE), Sept. 2007.

\bibitem{rs02}
{\sc M.~Raydan and B.~F. Svaiter}, {\em Relaxed steepest descent and
  {C}auchy-{B}arzilai-{B}orwein method}, Comput.~Optim.~Appl., 21 (2002),
  pp.~155--167.

\bibitem{Rockafellar70}
{\sc R.~T. Rockafellar}, {\em Convex analysis}, Princeton Univ. Press, 1970.

\bibitem{Berg09}
{\sc E.~van~den Berg, M.~P. Friedlander, G.~Hennenfent, F.~J. Herrmann,
  R.~Saab, and O.~Yilmaz}, {\em Algorithm 890: Sparco: A testing framework for
  sparse reconstruction}, ACM Trans. Math. Softw., 35 (2009), pp.~1--16.

\bibitem{Vandenberghe09}
{\sc L.~Vandenberghe}, {\em Gradient methods for nonsmooth problems (lecture
  note - spring 2009)}.
\newblock http://www.ee.ucla.edu/$\sim$vandenbe/ee236c.html.

\bibitem{Vonesch07}
{\sc C.~Vonesch and M.~Unser}, {\em Fast iterative thresholding algorithm for
  wavelet-regularized deconvolution}, in Proceedings of the SPIE Optics and
  Photonics 2007 Conference on Mathematical Methods: Wavelet {XII}, vol.~6701,
  San Diego, CA, 2007, pp.~1--5.

\bibitem{wright09a}
{\sc S.~J. Wright, R.~D. Nowak, and M.~A.~T. Figueiredo}, {\em Sparse
  reconstruction by separable approximation}, IEEE Trans. Signal Process., 57
  (2009), pp.~2479--2493.

\end{thebibliography}
\end{document}